\documentclass[10pt]{article}

\usepackage[T1]{fontenc}

\usepackage{amsmath}
\usepackage{amsthm}
\usepackage{amssymb}

\usepackage{url}
\usepackage{latexsym}
\usepackage{enumitem}

\usepackage{tikz}
\usetikzlibrary{arrows, automata, decorations.pathreplacing, fit, matrix, patterns, positioning, calc}
\usepackage{tikz-qtree}
\usepackage{xifthen} 

\usepackage[small,it,singlelinecheck=false]{caption}

\usepackage{titlefoot}
\usepackage[bf, small, center]{titlesec}

\usepackage{color}
\definecolor{lightgray}{rgb}{0.8, 0.8, 0.8}
\definecolor{darkgray}{rgb}{0.65, 0.65, 0.65}

\usepackage[bookmarks]{hyperref}
\hypersetup{
        colorlinks=true,
        linkcolor=black,
        anchorcolor=black,
        citecolor=black,
        urlcolor=black,
        pdfpagemode=UseThumbs,
        pdftitle={On the Lettericity of Paths},
        pdfsubject={Combinatorics},
        pdfauthor={Ferguson},
}

\newcounter{todocounter}

\theoremstyle{plain}
\newtheorem{theorem}{Theorem}

\newtheorem{lemma}[theorem]{Lemma}

\theoremstyle{definition}

\makeatletter
\newfont{\footsc}{cmcsc10 at 8truept}
\newfont{\footbf}{cmbx10 at 8truept}
\newfont{\footrm}{cmr10 at 10truept}
\renewenvironment{abstract}%
                {
                  \begin{list}{}%
                     {\setlength{\rightmargin}{1in}%
                      \setlength{\leftmargin}{1in}}%
                   \item[]\ignorespaces\begin{small}}%
                 {\end{small}\unskip\end{list}}

\makeatletter
\let\start@align@nopar\start@align
\let\start@gather@nopar\start@gather
\let\start@multline@nopar\start@multline
\long\def\start@align{\par\start@align@nopar}
\long\def\start@gather{\par\start@gather@nopar}
\long\def\start@multline{\par\start@multline@nopar}
\makeatother

\datefoot{\today}
\amssubj{05C75 (primary), 68R10 (secondary)}

\newpagestyle{main}[\small]{
        \headrule
        \sethead[\usepage][][]
        {\sc On the Lettericity of Paths}{}{\usepage}}

\title{\sc On the Lettericity of Paths}
\author{
	Robert Ferguson\\[-0.25ex]
	\small Department of Mathematics\\[-0.5ex]
	\small University of Florida\\[-0.5ex]
	\small Gainesville, Florida
}

\titleformat{\section}
        {\large\sc}
        {\thesection.}{1em}{}

\date{}

\begin{document}
\maketitle


\pagestyle{main}

\begin{abstract}
Verifying a conjecture of Petkov{\v{s}}ec, we prove that the lettericity of an n-vertex path is precisely $\left\lfloor \frac{n+4}{3}\right\rfloor$.
\end{abstract}

\section{Introduction}

The concept of lettericity was introduced in 2000 by Petkov{\v{s}}ec~\cite{petkovsek:letter-graphs-a:}. We begin by presenting his definitions. Let $\Sigma$ be a finite alphabet, and consider $D \subseteq \Sigma^2$, which we call the \emph{decoder}. Then for a word $w = w_1w_2\dots w_n$ with each $w_i \in \Sigma$, the \emph{letter graph} of $w$ is the graph $\Gamma_D(w)$ with  $V(\Gamma_D(w)) = \{1,2,\dots,n\}$ and $(i,j) \in E(\Gamma_D(w))$ if and only if $(w_i,w_j) \in D$.

If $\Sigma$ is an alphabet of size $k$, we say that $\Gamma_D(w)$ is a $k$-letter graph. For some graph $G$, the minimum $k$ such that a $G$ is a $k$-letter graph is known as the \emph{lettericity} of $G$, denoted $\ell(G)$.  Note that every finite graph is the letter graph of some word over some alphabet, and in particular the lettericity of a graph $G$ is at most $|V(G)|$.

Petkov{\v{s}}ec determined bounds or precise values for the lettericity of a number of different families of graphs, most notably threshold graphs, cycles, and paths. We focus our attention on paths, proving a conjecture of Pekov\v{s}ec's and giving a precise value for their lettericity. Before we begin our proof, however, we first introduce a few pieces of additional notation.

Given a letter graph $\Gamma_D(w)$ and some letter $a \in \Sigma$, we then say that $a$ \emph{encodes} the set of vertices that correspond to some instance of $a$ in the word. In particular, these vertices must form a clique if $(a,a) \in D$, and an anticlique otherwise. Further, given a graph $G$ such that $G = \Gamma_D(w)$, we say that $\Gamma_D(w)$ is a \emph{lettering} of $G$, and in particular an \emph{$r$-lettering} if $w$ uses an alphabet of size $r$.

\section{Lemmas}

We now establish a few lemmas necessary for the proof of our theorem. We begin with a simple but useful property of letter graphs.

\begin{lemma}
	If a letter graph $\Gamma_D(w)$ has some pair of vertices with indices $i$ and $k$ such that $i < k$ and $w_i = w_k$, and this pair is distinguished by some third vertex $j$ (that is, $j$ is adjacent to exactly one of $i$ and $k$), then it must be that $i < j < k$.
\end{lemma}

\begin{proof}
If it were the case that $j < i < k$ or that $i < k < j$, then it would have to be that the vertex $j$ of $\Gamma_D(w)$ is adjacent to either both of the vertices $i$ and $k$ or neither of them, depending on whether $(w_j,w_i) \in D$, in the first case, and $(w_i,w_j) \in D$ in the second. It thus must be that $i < j < k$.
\end{proof}

With this established, we now move on to examining matchings, a family of graphs for which the lettericity is already known.

\begin{lemma}
In any lettering of $rK_2$, no letter encodes more than 2 vertices.
\end{lemma}

\begin{proof}
Suppose there exists some lettering $\Gamma_D(W)$ of $rK_2$ with some letter $a$ that encodes at least three vertices of $\Gamma_D(w)$, say $i$, $j$, and $k$ with $i$ < $j$ < $k$. Our graph contains no cliques of size greater than 2, so these vertices form an anticlique. Each of these vertices is incident with a distinct edge, so there must be some vertex, say $x$, which is adjacent to $j$ but not $i$ or $k$. Then, by Lemma 1 it must be that $i < x < j$, but also that $j < x < k$. This is a clear contradiction, so no such lettering exists.
\end{proof}

With this lemma it is not difficult to show that $\ell(rK_2) = r$, a property noted by Petkov{\v{s}}ec and explicitly proven by Alecu, Lozin and De Werra~\cite{alecu:the-micro-world:}. Our objective here is the additional property our next lemma presents.

\begin{lemma}
In every $r$-lettering of $rK_2$, each letter encodes the two vertices of a $K_2$.
\end{lemma}

\begin{proof}
That each letter encodes exactly two vertices follows easily from Lemma 2. Now suppose $rK_2$ has some other $r$-lettering, and choose $a$ to be the earliest occurring letter that encodes an anticlique. In particular, suppose it first occurs at index $i$. Then vertex $i$ is adjacent to some vertex encoded by a different letter, say $b$. Then $b$ also encodes an anticlique, and by our choice of $a$, both of the vertices it encodes must lie after $i$ in the word. They then must both be adjacent to $i$; since $rK_2$ has no vertices of degree two, no such $r$-lettering exists.
\end{proof}

This gives us the opportunity to count the number of $r$-letterings of $rK_2$; in particular it can be easily calculated that there are $2r!/2^r$ possible choices for $w$. It is, however, more difficult to determine the number of possible decoders for each word, a problem which is beyond the scope of this paper.

\section{Theorem and Proof}

We now prove our main result.

\begin{theorem}
For $n \geq 3$, the lettericity of $P_n$ is $\left\lfloor \frac{n+4}{3}\right\rfloor$.
\end{theorem}

\begin{proof}
We begin with the lower bound; it suffices to examine a path $P_n$ with $n = 3r+1$, which our theorem claims has lettericity $r+1$. Label the vertices of $P_n$ as $i_1,i_2,...,i_{3r+1}$ so that $i_i \sim i_2 \sim ... \sim i_{3r+1}$, and consider its subgraph $P_n[i_2,i_3,i_5,i_6,...,i_{3r-1},i_{3r}] = rK_2$.

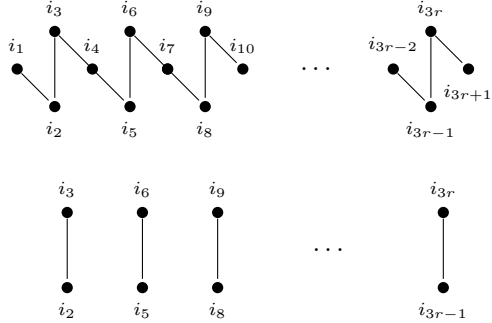
\begin{center}
\begin{tikzpicture}[->,>=stealth',shorten >=1pt,auto,node distance=1.5cm,
    main node/.style={thick,circle,draw,font=\sffamily\Large}]
	\begin{scope}
	\tikz [tight/.style={inner sep=1pt}];
		\node[label=above:{\scriptsize $i_1$}] (0) at (.5,.5)[circle,fill,inner sep=1.5pt]{};
		\node[label=above:{\scriptsize $i_3$}] (A) at (1,1)[circle,fill,inner sep=1.5pt]{};
		\node[label=above:{\scriptsize $i_4$}] (A') at (1.5,.5)[circle,fill,inner sep=1.5pt]{};
		\node[label=below:{\scriptsize $i_2$}] (B) at (1,0)[circle,fill,inner sep=1.5pt]{};
		\node[label=above:{\scriptsize $i_6$}] (C) at (2,1)[circle,fill,inner sep=1.5pt]{};
		\node[label=above:{\scriptsize $i_7$}] (C') at (2.5,.5)[circle,fill,inner sep=1.5pt]{};
		\node[label=below:{\scriptsize $i_5$}] (D) at (2,0)[circle,fill,inner sep=1.5pt]{};
		\node[label=above:{\scriptsize $i_9$}] (E) at (3,1)[circle,fill,inner sep=1.5pt]{};
		\node[label=above:{\scriptsize }] (E') at (2.5,.5)[circle,fill,inner sep=1.5pt]{};
		\node[label=below:{\scriptsize $i_8$}] (F) at (3,0)[circle,fill,inner sep=1.5pt]{};
		\node[label=above:{\scriptsize $i_{10}$}] (G) at (3.5,.5)[circle,fill,inner sep=1.5pt]{};
		\node[label=above:{\scriptsize $i_{3r-2}$}] (H) at (5.5,.5)[circle,fill,inner sep=1.5pt]{};
		\node[label=above:{\scriptsize $i_{3r}$}] (I) at (6,1)[circle,fill,inner sep=1.5pt]{};
		\node[label=below:{\scriptsize $i_{3r-1}$}] (J) at (6,0)[circle,fill,inner sep=1.5pt]{};
		\node[label=below:{\scriptsize $i_{3r+1}$}] (J') at (6.5,0.5)[circle,fill,inner sep=1.5pt]{};
		
		\path[-]
		(0) edge node[label distance=1pt, right] {} (B)
		(A) edge node[label distance=1pt, right] {} (B)
		(A) edge node[label distance=1pt, right] {} (D)
		(C) edge node[label distance=1pt, above right] {} (D)
		(C) edge node[label distance=1pt, above right] {} (F)
		(E) edge node[label distance=1pt, below] {} (F)
		(E) edge node[label distance=1pt, below] {} (G)
		(H) edge node[label distance=1pt, below] {} (J)
		(I) edge node[label distance=1pt, above right] {} (J)
		(I) edge node[label distance=1pt, above right] {} (J');
		
    	\path (G) -- node[auto=false]{\ldots} (H);

	\end{scope}
\end{tikzpicture}

\begin{tikzpicture}[->,>=stealth',shorten >=1pt,auto,node distance=1.5cm,
    main node/.style={thick,circle,draw,font=\sffamily\Large}]
	\begin{scope}
	\tikz [tight/.style={inner sep=1pt}];
		\node[label=above:{\scriptsize }] (0) at (.5,.5){};
		\node[label=above:{\scriptsize $i_3$}] (A) at (1,1)[circle,fill,inner sep=1.5pt]{};
		\node[label=below:{\scriptsize $i_2$}] (B) at (1,0)[circle,fill,inner sep=1.5pt]{};
		\node[label=above:{\scriptsize $i_6$}] (C) at (2,1)[circle,fill,inner sep=1.5pt]{};
		\node[label=below:{\scriptsize $i_5$}] (D) at (2,0)[circle,fill,inner sep=1.5pt]{};
		\node[label=above:{\scriptsize $i_9$}] (E) at (3,1)[circle,fill,inner sep=1.5pt]{};
		\node[label=below:{\scriptsize $i_8$}] (F) at (3,0)[circle,fill,inner sep=1.5pt]{};
		\node[label=above:{\scriptsize $i_{3r}$}] (I) at (6,1)[circle,fill,inner sep=1.5pt]{};
		\node[label=below:{\scriptsize $i_{3r-1}$}] (J) at (6,0)[circle,fill,inner sep=1.5pt]{};
		\node[label=below:{\scriptsize }] (J') at (6.5,0.5){};
		
		\path[-]
		(A) edge node[label distance=1pt, right] {} (B)
		(C) edge node[label distance=1pt, above right] {} (D)
		(E) edge node[label distance=1pt, below] {} (F)
		(I) edge node[label distance=1pt, above right] {} (J);
		
    	\path (G) -- node[auto=false]{\ldots} (H);

	\end{scope}
\end{tikzpicture}
\end{center}

Suppose $P_n$ has some $r$-lettering $\Gamma_D(w)$. Then $rK_2$ is a letter graph for some subword of $w$, which must still require an alphabet of size $r$. By Lemma 3, this is only possible if each letter is assigned to a distinct adjacent pair. The vertices encoded by each letter thus form cliques; they then do so in $\Gamma_D(w)$ as well. As $\Gamma_D(w)$ contains no cliques of size larger than 2, no such lettering exists.

The upper bound has already been established by Petkov{\v{s}}ek, but here we show how this bound is obtained from an $r+1$-lettering of $rK_2$. Take an ordering of the adjacent pairs in $rK_2$, and take the lettering of $rK_2$ which assigns to the $i$th adjacent pair the letters $i,i+1$. Since we have $r$ pairs, this requires $r+1$ letters in total.

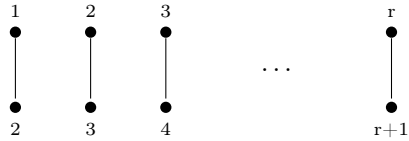
\begin{center}
	\begin{tikzpicture}[->,>=stealth',shorten >=1pt,auto,node distance=1.5cm,
    main node/.style={thick,circle,draw,font=\sffamily\Large}]
	\begin{scope}
	\tikz [tight/.style={inner sep=1pt}];
		\node[label=above:{\scriptsize 1}] (A) at (1,1)[circle,fill,inner sep=1.5pt]{};
		\node[label=below:{\scriptsize 2}] (B) at (1,0)[circle,fill,inner sep=1.5pt]{};
		\node[label=above:{\scriptsize 2}] (C) at (2,1)[circle,fill,inner sep=1.5pt]{};
		\node[label=below:{\scriptsize 3}] (D) at (2,0)[circle,fill,inner sep=1.5pt]{};
		\node[label=above:{\scriptsize 3}] (E) at (3,1)[circle,fill,inner sep=1.5pt]{};
		\node[label=below:{\scriptsize 4}] (F) at (3,0)[circle,fill,inner sep=1.5pt]{};
		\node (G) at (4,.5){};
		\node (H) at (5,.5){};
		\node[label=above:{\scriptsize r}] (I) at (6,1)[circle,fill,inner sep=1.5pt]{};
		\node[label=below:{\scriptsize r+1}] (J) at (6,0)[circle,fill,inner sep=1.5pt]{};
		
		\path[-]
		(A) edge node[label distance=1pt, right] {} (B)
		(C) edge node[label distance=1pt, above right] {} (D)
		(E) edge node[label distance=1pt, below] {} (F)
		(I) edge node[label distance=1pt, above right] {} (J);
		
    	\path (G) -- node[auto=false]{\ldots} (H);

	\end{scope}
\end{tikzpicture}
\end{center}

The graph above is the letter graph of the word $ 21324354...r(r-1)(r+1)r$ with the decoder $D = \{(2,1),(3,2),\cdots(r+1,r)\}$.

We now add $r-1$ new vertices, giving the $j$th new vertex the label $j+1$ and connecting it to the vertex in the $j$th pair labelled $j$ and the vertex in the $j+1$st pair labeled $j+2$. Finally, we add a vertex labeled 1 adjacent to the vertex in the first pair labeled 2 and a vertex labeled $r+1$ adjacent to the vertex in the last pair labeled $r$.

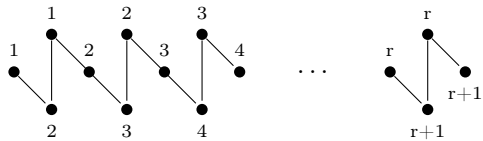
\begin{center}
\begin{tikzpicture}[->,>=stealth',shorten >=1pt,auto,node distance=1.5cm,
    main node/.style={thick,circle,draw,font=\sffamily\Large}]
	\begin{scope}
	\tikz [tight/.style={inner sep=1pt}];
		\node[label=above:{\scriptsize 1}] (0) at (.5,.5)[circle,fill,inner sep=1.5pt]{};
		\node[label=above:{\scriptsize 1}] (A) at (1,1)[circle,fill,inner sep=1.5pt]{};
		\node[label=above:{\scriptsize 2}] (A') at (1.5,.5)[circle,fill,inner sep=1.5pt]{};
		\node[label=below:{\scriptsize 2}] (B) at (1,0)[circle,fill,inner sep=1.5pt]{};
		\node[label=above:{\scriptsize 2}] (C) at (2,1)[circle,fill,inner sep=1.5pt]{};
		\node[label=above:{\scriptsize 3}] (C') at (2.5,.5)[circle,fill,inner sep=1.5pt]{};
		\node[label=below:{\scriptsize 3}] (D) at (2,0)[circle,fill,inner sep=1.5pt]{};
		\node[label=above:{\scriptsize 3}] (E) at (3,1)[circle,fill,inner sep=1.5pt]{};
		\node[label=below:{\scriptsize 4}] (F) at (3,0)[circle,fill,inner sep=1.5pt]{};
		\node[label=above:{\scriptsize 4}] (G) at (3.5,.5)[circle,fill,inner sep=1.5pt]{};
		\node[label=above:{\scriptsize r}] (H) at (5.5,.5)[circle,fill,inner sep=1.5pt]{};
		\node[label=above:{\scriptsize r}] (I) at (6,1)[circle,fill,inner sep=1.5pt]{};
		\node[label=below:{\scriptsize r+1}] (J) at (6,0)[circle,fill,inner sep=1.5pt]{};
		\node[label=below:{\scriptsize r+1}] (J') at (6.5,0.5)[circle,fill,inner sep=1.5pt]{};
		
		\path[-]
		(0) edge node[label distance=1pt, right] {} (B)
		(A) edge node[label distance=1pt, right] {} (B)
		(A) edge node[label distance=1pt, right] {} (D)
		(C) edge node[label distance=1pt, above right] {} (D)
		(C) edge node[label distance=1pt, above right] {} (F)
		(E) edge node[label distance=1pt, below] {} (F)
		(E) edge node[label distance=1pt, below] {} (G)
		(H) edge node[label distance=1pt, below] {} (J)
		(I) edge node[label distance=1pt, above right] {} (J)
		(I) edge node[label distance=1pt, above right] {} (J');
		
    	\path (G) -- node[auto=false]{\ldots} (H);

	\end{scope}
\end{tikzpicture}
\end{center}

This new graph is the letter graph of the word $21321432543\dots(r+1)r(r-1)(r+1)r$ with the same decoder $D = \{(2,1),(3,2),...(r+1,r)\}$. This gives us a path on $3r+1$ vertices; to obtain a path on $3r$ vertices we remove the first instance of 1 in our word, and to obtain a path on $3r-1$ we additionally remove the last instance of $r+1$.
 \end{proof}

%
%
%
%
%


\end{document}